\documentclass[12pt, reqno]{amsart}
\usepackage{amsmath, amsthm, amscd, amsfonts, amssymb, graphicx, color}
\usepackage[bookmarksnumbered, colorlinks, plainpages]{hyperref}
\hypersetup{colorlinks=true,linkcolor=red, anchorcolor=green,
citecolor=cyan, urlcolor=red, filecolor=magenta, pdftoolbar=true}

\newtheorem{theorem}{Theorem}[section]

\newtheorem{proposition}[theorem]{Proposition}
\newtheorem{corollary}[theorem]{Corollary}
\theoremstyle{definition}

\theoremstyle{remark}
\newtheorem{remark}[theorem]{Remark}
\numberwithin{equation}{section}

\begin{document}

\setcounter{page}{1}

\title[Numerical radius inequalities concerning with algebraic norms]
{Numerical radius inequalities concerning with algebraic norms}

\author[A. Zamani, M.S. Moslehian, Q. Xu \MakeLowercase{and} C. Fu]
{A. Zamani$^1$, M. S. Moslehian$^2$, Q. Xu$^3$ \MakeLowercase{and} C. Fu$^4$}

\address{$^1$Department of Mathematics, Farhangian University, Tehran, Iran}
\email{zamani.ali85@yahoo.com}

\address{$^2$Department of Pure Mathematics, Ferdowsi University of Mashhad,
Center of Excellence in Analysis on Algebraic Structures (CEAAS),
P. O. Box 1159, Mashhad 91775, Iran}
\email{moslehian@um.ac.ir, moslehian@yahoo.com}

\address{$^3$Department of Mathematics, Shanghai Normal University, Shanghai 200234, PR China}
\email{qingxiang\_xu@126.com}

\address{$^4$Department of Mathematics, Shanghai Normal University, Shanghai 200234, PR China}
\email{fchlixue@163.com}

\subjclass[2010]{Primary 47A12; Secondary 47A30.}

\keywords{Numerical radius, usual operator norm, inequality, Cartesian decomposition.}
%%%%%%%%%%%%%%%%%%%%%%%%%%%%%%%%%
%%%%%%%%%%%%%%%%%%%%%%%%%%%%%%%%%
\begin{abstract}
We give an expression for a generalized numerical radius of Hilbert space operators
and then apply it to obtain upper and lower bounds for the generalized numerical radius.
We also establish some generalized numerical radius inequalities involving the product of two operators.
Applications of our inequalities are also provided.
\end{abstract} \maketitle
%%%%%%%%%%%%%%%%%%%%%%%%%%%%%%%%%%%%%%%%
%%%%%%%%%%%%%%%%%%%%%%%%%%%%%%%%%%%%%%%%
\section{Introduction and preliminaries}

Let $(\mathcal{H}, \langle \cdot,  \cdot\rangle)$ be a complex Hilbert space
and let $\mathbb{B}(\mathcal{H})$ be the algebra of all bounded linear operators on $\mathcal{H}$.
For $T\in\mathbb{B}(\mathcal{H})$, let $\|T\|=\sup\{\|Tx\|: \|x\| = 1\}$ and
$w(T)= \sup\{|\langle Tx,  x\rangle|: \|x\| = 1\}$
denote the usual operator norm and the numerical radius of $T$, respectively.
It is easy to see that $w(\cdot)$ defines a norm on $\mathbb{B}(\mathcal{H})$,
which is equivalent to the usual operator norm $\|\cdot\|$.
Namely, for $T\in \mathbb{B}(\mathcal{H})$, we have
\begin{align}\label{111}
\frac{1}{2}\|T\| \leq w(T) \leq \|T\|.
\end{align}
The first inequality becomes equality if $T$ is square-zero (i.e., $T^2 = 0$)
and the second inequality becomes equality if $T$ is normal (see, e.g., \cite{G.R}).

Over the years, double inequality (\ref{111}) has been improved to various sharp inequalities.
For example, Kittaneh \cite{K} refined the right-hand side of (\ref{111}) by proving that
\begin{align}\label{222}
w(T) \leq \frac{\sqrt{2}}{2}\sqrt{\|TT^* + T^*T\|}.
\end{align}
In another vein, Dragomir \cite{D} used Buzano inequality to improve the right-hand side of (\ref{111}) by showing that
\begin{align}\label{333}
w(T) \leq \frac{\sqrt{2}}{2}\sqrt{\|T\|^2 + w(T^2)}.
\end{align}
Some other interesting numerical radius inequalities improving inequalities (\ref{111})
can be found in \cite{G.R, H.K.S.2, S.M.Y, SDM, Z.1, Z.2}.

Every operator $T\in \mathbb{B}(\mathcal{H})$ can be represented as $T =\,{\rm Re}(T) + i\,{\rm Im}(T)$,
the Cartesian decomposition, where ${\rm Re}(T)= \frac{T+T^*}{2}$ and ${\rm Im}(T)= \frac{T-T^*}{2i}$
are the real and imaginary parts of $T$, respectively.
It is well-known (see, e.g., \cite{Y}) that
\begin{align*}
w(T) = \displaystyle{\sup_{\theta \in \mathbb{R}}}\big\|{\rm Re}(e^{i\theta}T)\big\|.
\end{align*}
Also, it has been shown in \cite{K.M.Y} that for $\alpha, \beta \in \mathbb{R}$,
\begin{align}\label{0101}
w(T) = \displaystyle{\sup_{\alpha^2 + \beta^2 = 1}}
\Big\|\alpha {\rm Re}(T) + \beta {\rm Im}(T)\Big\|.
\end{align}

Let $N(\cdot)$ be a norm on $\mathbb{B}(\mathcal{H})$.
The norm $N(\cdot)$ is said to be an algebra norm if $N(TS) \leq N(T)N(S)$ for every $T, S\in\mathbb{B}(\mathcal{H})$,
and is called self-adjoint if $N(T^*) = N(T)$ for every $T\in\mathbb{B}(\mathcal{H})$.
For $T\in\mathbb{B}(\mathcal{H})$, we recall from \cite{A.K.2} the following generalization of the numerical radius:
\begin{align*}
w_{N}(T) = \displaystyle{\sup_{\theta \in \mathbb{R}}}
N\big({\rm Re}(e^{i\theta}T)\big).
\end{align*}
In particular, by taking $\theta = 0$ and $\theta = \frac{\pi}{2}$,
we have $N\big({\rm Re}(T)\big)\leq w_{N}(T)$ and $N\big({\rm Im}(T)\big)\leq w_{N}(T)$.
Abu-Omar and Kittaneh \cite{A.K.2} showed that $w_{N}(\cdot)$ is a self-adjoint norm
and $\frac{1}{2}N(T)\leq w_{N}(T)$ for all $T\in\mathbb{B}(\mathcal{H})$.
Also, if $T$ is self-adjoint, then $w_{N}(T)= N(T)$.
Furthermore, if $N(\cdot)$ is a self-adjoint norm on $\mathbb{B}(\mathcal{H})$, then $w_{N}(T)\leq N(T)$.
Therefore, for a self-adjoint norm $N(\cdot)$ on $\mathbb{B}(\mathcal{H})$, we have
\begin{align}\label{666}
\frac{1}{2}N(T)\leq w_{N}(T)\leq N(T).
\end{align}

The paper is organized as follows.

In Section 2, inspired by the numerical radius inequalities in \cite{K.M.Y},
for a given norm $N(\cdot)$ on $\mathbb{B}(\mathcal{H})$, we present
an expression of $w_{N}(\cdot)$, which generalizes equality (\ref{0101}),
and then apply it to obtain upper and lower bounds for $w_{N}(\cdot)$.
Further, following \cite{A.K.1, KIT1}, we obtain some generalized numerical radius inequalities involving the product of two operators.
The last section will present a refinement of the second inequality (\ref{111}),
which also refines inequalities (\ref{222}) and (\ref{333}).
%%%%%%%%%%%%%%%%%%%%%%%%%%%
%%%%%%%%%%%%%%%%%%%%%%%%%%%
\section{Main result}
We start this section by finding an upper bound for the generalized numerical radius as follows.
%%%%%%%%%%%%%%%%%%%%%%%%%%%
\begin{theorem}\label{T.3.3}
Let $T\in\mathbb{B}(\mathcal{H})$. Then
\begin{align*}
w_{N}(T) \leq \displaystyle{\inf_{\varphi \in \mathbb{R}}}\sqrt{N^2\big(\,{\rm Re}(e^{i\varphi}T)\big)
+ N^2\big(\,{\rm Im}(e^{i\varphi}T)\big)}.
\end{align*}
\end{theorem}
\begin{proof}
Let $\theta\in \mathbb{R}$. Put $\alpha = \cos \theta$ and $\beta = -\sin \theta$.
We have
\begingroup\makeatletter\def\f@size{10}\check@mathfonts
\begin{align*}
{\rm Re}(e^{i(\theta)}T) = \frac{\big(\cos(\theta) + i\sin (\theta)\big)T
+ \big(\cos(\theta) - i\sin (\theta)\big)T^*}{2} = \alpha\,{\rm Re}(T) + \beta\,{\rm Im}(T).
\end{align*}
\endgroup
Hence
\begin{align}\label{T.1.3}
w_{N}(T) = \displaystyle{\sup_{\theta \in \mathbb{R}}}\,N\Big({\rm Re}(e^{i\theta}T)\Big)
= \displaystyle{\sup_{\alpha^2 + \beta^2 = 1,\,\alpha , \beta \in \mathbb{R}}}
N\Big(\alpha {\rm Re}(T) + \beta {\rm Im}(T)\Big).
\end{align}
Now, let $\varphi\in \mathbb{R}$. For $\alpha , \beta \in \mathbb{R}$,
by employing (\ref{T.1.3}) and the Cauchy--Schwarz inequality, we have
\begin{align*}
w_{N}(T) &= \displaystyle{\sup_{\alpha^2 + \beta^2 = 1}}
N\Big(\alpha {\rm Re}(e^{i\varphi}T) + \beta {\rm Im}(e^{i\varphi}T)\Big)
\\& \leq \displaystyle{\sup_{\alpha^2 + \beta^2 = 1}}
|\alpha| N\Big({\rm Re}(e^{i\varphi}T)\Big) + |\beta|N\Big( {\rm Im}(e^{i\varphi}T)\Big)
\\& \leq \displaystyle{\sup_{\alpha^2 + \beta^2 = 1}}
\sqrt{\alpha^2 + \beta^2}\sqrt{N^2\Big({\rm Re}(e^{i\varphi}T)\Big) + N^2\Big( {\rm Im}(e^{i\varphi}T)\Big)}
\\& = \sqrt{N^2\big(\,{\rm Re}(e^{i\varphi}T)\big)
+ N^2\big(\,{\rm Im}(e^{i\varphi}T)\big)}.
\end{align*}
Thus
\begin{align*}
w_{N}(T) \leq \displaystyle{\inf_{\varphi \in \mathbb{R}}}\sqrt{N^2\big(\,{\rm Re}(e^{i\varphi}T)\big)
+ N^2\big(\,{\rm Im}(e^{i\varphi}T)\big)}.
\end{align*}
\end{proof}
%%%%%%%%%%%%%%%%%%%%%%%%%%%%
%%%%%%%%%%%%%%%%%%%%%%%%%%%%
\begin{remark}\label{R.4.3}
Let $T\in\mathbb{B}(\mathcal{H})$. Considering
the Cartesian decomposition of $T$, we have
$w_{N}(T)\leq w_{N}\big(\,{\rm Re}(T)\big) + w_{N}\big(\,{\rm Im}(T)\big)$.
It follows from $w_{N}\big(\,{\rm Re}(T)\big)= N\big(\,{\rm Re}(T)\big)$ and
$w_{N}\big(\,{\rm Im}(T)\big)= N\big(\,{\rm Im}(T)\big)$ that
\begin{align}\label{I.1.R.4.3}
w_{N}(T)\leq N\big(\,{\rm Re}(T)\big) + N\big(\,{\rm Im}(T)\big).
\end{align}
On the other hand, by Theorem \ref{T.3.3} and the arithmetic-geometric mean inequality, we have
\begin{align*}
\displaystyle{\inf_{\varphi \in \mathbb{R}}}\sqrt{N^2\big(\,{\rm Re}(e^{i\varphi}T)\big)
+ N^2\big(\,{\rm Im}(e^{i\varphi}T)\big)}
& \leq \sqrt{N^2\big(\,{\rm Re}(T)\big)
+ N^2\big(\,{\rm Im}(T)\big)}
\\& \leq N\big(\,{\rm Re}(T)\big) + N\big(\,{\rm Im}(T)\big),
\end{align*}
and hence Theorem \ref{T.3.3} refines inequality (\ref{I.1.R.4.3}).
In particular, when $N(\cdot)$ is the usual operator norm, the inequality in Theorem \ref{T.3.3}
becomes
\begin{align}\label{I.2.R.4.3}
w(T) \leq \displaystyle{\inf_{\varphi \in \mathbb{R}}}\sqrt{\big\|\,{\rm Re}(e^{i\varphi}T)\big\|^2
+ \big\|\,{\rm Im}(e^{i\varphi}T)\big\|^2},
\end{align}
which actually refines the inequality
\begin{align}\label{I.3.R.4.3}
w(T)\leq \big\|\,{\rm Re}(T)\big\| + \big\|\,{\rm Im}(T)\big\|.
\end{align}
The following example shows that inequality (\ref{I.2.R.4.3}) is a nontrivial improvement of  inequality (\ref{I.3.R.4.3}).
Consider $T = \begin{bmatrix}
1 & 1 \\
0 & 0
\end{bmatrix}$.
Simple computations show that $w(T) = \frac{1 + \sqrt{2}}{2}$, $\big\|\,{\rm Re}(T)\big\| = \frac{\sqrt{3 + 2\sqrt{2}}}{2}$
and $\big\|\,{\rm Im}(T)\big\| = \frac{1}{2}$. Furthermore, for every $\varphi \in \mathbb{R}$ one can easily observe that
\begin{align*}
\big\|\,{\rm Re}(e^{i\varphi}T)\big\|^2 = \frac{1 + 2\cos^2 \varphi}{4} + \frac{\sqrt{\cos^2 \varphi + \cos^4 \varphi}}{2},
\end{align*}
\begin{align*}
\big\|\,{\rm Im}(e^{i\varphi}T)\big\|^2 = \frac{1 + 2\sin^2 \varphi}{4} + \frac{\sqrt{\sin^2 \varphi + \sin^4 \varphi}}{2}
\end{align*}
and
\begingroup\makeatletter\def\f@size{10}\check@mathfonts
\begin{align*}
\displaystyle{\inf_{\varphi \in \mathbb{R}}}\sqrt{\big\|\,{\rm Re}(e^{i\varphi}T)\big\|^2
+ \big\|\,{\rm Im}(e^{i\varphi)}T)\big\|^2}
&= \displaystyle{\inf_{\varphi \in \mathbb{R}}}
\sqrt{ 1 + \frac{1}{2}\Big(\sqrt{\cos^2 \varphi + \cos^4 \varphi} + \sqrt{\sin^2 \varphi + \sin^4 \varphi}\Big)}
\\&= \sqrt{1 + \frac{\sqrt{2}}{2}}.
\end{align*}
\endgroup
Thus
\begin{align*}
w(T) &= \frac{1 + \sqrt{2}}{2}
\\& < \displaystyle{\inf_{\varphi \in \mathbb{R}}}\sqrt{\big\|\,{\rm Re}(e^{i\varphi}T)\big\|^2
+ \big\|\,{\rm Im}(e^{i\varphi}T)\big\|^2} = \sqrt{1 + \frac{\sqrt{2}}{2}}
\\& < \big\|\,{\rm Re}(T)\big\| + \big\|\,{\rm Im}(T)\big\| = 1 + \frac{\sqrt{2}}{2}.
\end{align*}
\end{remark}
%%%%%%%%%%%%%%%%%%%%%%%%%%%%%
%%%%%%%%%%%%%%%%%%%%%%%%%%%%%%
In the next theorem, we give a lower bound for the generalized
numerical radius of operators.
%%%%%%%%%%%%%%%%%%%%%%%%%%%%%%%
%%%%%%%%%%%%%%%%%%%%%%%%%%%%%%
\begin{theorem}\label{T.5.3}
Let $T\in\mathbb{B}(\mathcal{H})$ and let $N(\cdot)$ be an algebra norm on $\mathbb{B}(\mathcal{H})$.
Then
\begin{align*}
\frac{N\big(TT^* + T^*T\big)}{4}
+ \frac{1}{2}\displaystyle{\sup_{\varphi \in \mathbb{R}}}\Big|N^2\big(\,{\rm Re}(e^{i\varphi}T)\big) - N^2\big(\,{\rm Im}(e^{i\varphi}T)\big)\Big|
\leq w^2_{N}(T).
\end{align*}
\end{theorem}
\begin{proof}
Let $\varphi\in \mathbb{R}$. An easy calculation shows that
\begin{align}\label{I.1.T.5.3}
{\rm Re}^2(e^{i\varphi}T) + {\rm Im}^2(e^{i\varphi}T) = \frac{1}{2}\big(TT^* + T^*T \big).
\end{align}
Also, we have $w_{N}(T)\geq \max\big\{N\big({\rm Re}(e^{i\varphi}T)\big), N\big({\rm Im}(e^{i\varphi}T)\big)\big\}$. Thus
\begin{align*}
w^2_{N}(T) &\geq \max\Big\{N^2\big({\rm Re}(e^{i\varphi}T)\big), N^2\big({\rm Im}(e^{i\varphi}T)\big)\Big\}
\\& = \frac{N^2\big({\rm Re}(e^{i\varphi}T)\big) + N^2\big({\rm Im}(e^{i\varphi}T)\big)}{2}
+ \frac{\Big|N^2\big({\rm Re}(e^{i\varphi}T)\big) - N^2\big({\rm Im}(e^{i\varphi}T)\big)\Big|}{2}
\\& \geq \frac{N\big({\rm Re}^2(e^{i\varphi}T)\big) + N\big({\rm Im}^2(e^{i\varphi}T)\big)}{2}
+ \frac{\Big|N^2\big({\rm Re}(e^{i\varphi}T)\big) - N^2\big({\rm Im}(e^{i\varphi}T)\big)\Big|}{2}
\\& \qquad \qquad \qquad \qquad \qquad \qquad \qquad \qquad \qquad \Big(\mbox{since} \,N(\cdot) \, \mbox{is an algebra norm}\Big)
\\& \geq \frac{N\Big({\rm Re}^2(e^{i\varphi}T) + {\rm Im}^2(e^{i\varphi}T)\Big)}{2}
+ \frac{\Big|N^2\big({\rm Re}(e^{i\varphi}T)\big) - N^2\big({\rm Im}(e^{i\varphi}T)\big)\Big|}{2}
\\& = \frac{N\big(TT^* + T^*T\big)}{4} + \frac{\Big|N^2\big({\rm Re}(e^{i\varphi}T)\big) - N^2\big({\rm Im}(e^{i\varphi}T)\big)\Big|}{2}
\qquad \qquad \big(\mbox{by (\ref{I.1.T.5.3})}\big).
\end{align*}
Hence
\begin{align*}
\frac{N\big(TT^* + T^*T\big)}{4}
+ \frac{1}{2}\displaystyle{\sup_{\varphi \in \mathbb{R}}}\Big|N^2\big(\,{\rm Re}(e^{i\varphi}T)\big) - N^2\big(\,{\rm Im}(e^{i\varphi}T)\big)\Big|
\leq w^2_{N}(T).
\end{align*}
\end{proof}
%%%%%%%%%%%%%%%%%%%%%%%%%%%%%%%%%%
%%%%%%%%%%%%%%%%%%%%%%%%%%%%%%%%%%
\begin{remark}\label{R.6.3}
When $N(\cdot)$ is the usual operator norm, the inequality in Theorem \ref{T.5.3} becomes
\begin{align}\label{I.1.R.6.3}
\frac{\big\|\,TT^* + T^*T\,\big\|}{4}
+ \frac{1}{2}\displaystyle{\sup_{\varphi \in \mathbb{R}}}\Big|\,\|\,{\rm Re}(e^{i\varphi}T)\|^2 - \|\,{\rm Im}(e^{i\varphi}T)\|^2\Big| \leq w^2(T).
\end{align}
This inequality refines the inequality $\frac{\big\|\,TT^* + T^*T\,\big\|}{4}\leq w^2(T)$ in \cite{K}.
\end{remark}
%%%%%%%%%%%%%%%%%%%%%%%%%%%%%%%%
%%%%%%%%%%%%%%%%%%%%%%%%%%%%%%%%%
Let $N(\cdot)$ be a norm on $\mathbb{B}(\mathcal{H})$.
If $N(\cdot)$ is a self-adjoint algebra norm,
it follows directly from (\ref{666}) that for every $T, S\in\mathbb{B}(\mathcal{H})$,
\begin{align}\label{3.00}
w_{N}(TS) \leq N(TS) \leq N(T)N(S) \leq 2N(T)w_{N}(S) \leq 4w_{N}(T)w_{N}(S).
\end{align}
Next, by adopting some ideas of \cite{A.K.1, KIT1}, we give some inequalities involving
the generalized numerical radius of the product of two operators, which refine inequalities (\ref{3.00}).

In order to achieve our aim, we need the following result.
%%%%%%%%%%%%%%%%%%%%%%%
%%%%%%%%%%%%%%%%%%%%%%%
\begin{theorem}\label{T.1.2}
Let $T, S\in\mathbb{B}(\mathcal{H})$. If $N(\cdot)$ is an algebra norm, then
\begin{align*}
w_{N}(TS \pm  ST^*) \leq w_{N}(S)\big(N(T) + N(T^*)\big).
\end{align*}
In particular, if $N(\cdot)$ is a self-adjoint algebra norm, then
\begin{align*}
w_{N}(TS \pm  ST^*) \leq 2w_{N}(S)N(T).
\end{align*}
\end{theorem}
\begin{proof}
Let $\theta \in \mathbb{R}$. We have
\begin{align*}
N\Big({\rm Re}\big(e^{i\theta}(TS + ST^*)\big)\Big)
&= N\left(\frac{e^{i\theta}(TS + ST^*) + e^{-i\theta}(S^*T^* + TS^*)}{2}\right)
\\& = N\left(T\,\frac{e^{i\theta}S + e^{-i\theta}S^*}{2} + \frac{e^{i\theta}S + e^{-i\theta}S^*}{2}\,T^*\right)
\\& = N\Big(T\,{\rm Re}(e^{i\theta}S) + {\rm Re}(e^{i\theta}S)\,T^*\Big)
\\& \leq N\big(T\,{\rm Re}(e^{i\theta}S)\big) + N\big({\rm Re}(e^{i\theta}S)\,T^*\big)
\\& \leq N(T)N\big({\rm Re}(e^{i\theta}S)\big) + N\big({\rm Re}(e^{i\theta}S)\big)N(T^*)
\\&\qquad\qquad\qquad\qquad\Big(\mbox{since $N(\cdot)$ is an algebra norm}\Big)
\\& = N\big({\rm Re}(e^{i\theta}S)\big)\big(N(T) + N(T^*)\big)
\\& \leq w_{N}(S)\big(N(T) + N(T^*)\big).
\end{align*}
Now, by taking the supremum over all $\theta \in \mathbb{R}$, we obtain
\begin{align*}
\displaystyle{\sup_{\theta \in \mathbb{R}}}\,N\Big({\rm Re}\big(e^{i\theta}(TS + ST^*)\big)\Big) \leq w_{N}(S)\big(N(T) + N(T^*)\big),
\end{align*}
and hence
\begin{align}\label{I.2.T.1.2}
w_{N}(TS + ST^*) \leq w_{N}(S)\big(N(T) + N(T^*)\big).
\end{align}
Furthermore, by replacing $T$ in (\ref{I.2.T.1.2}) by $iT$, we arrive at
\begin{align}\label{I.3.T.1.2}
w_{N}\big((iT)S + S(iT)^*\big) \leq w_{N}(S)\big(N(iT) + N((iT)^*)\big).
\end{align}
Since $w_{N}\big((iT)S + S(iT)^*\big) = w_{N}(TS - ST^*)$, $N(iT) = N(T)$ and $N((iT)^*) = N(T^*)$, by virtue of (\ref{I.3.T.1.2}), we get
\begin{align}\label{I.4.T.1.2}
w_{N}(TS - ST^*) \leq w_{N}(S)\big(N(T) + N(T^*)\big).
\end{align}
Now from (\ref{I.2.T.1.2}) and (\ref{I.4.T.1.2}) it follows that $w_{N}(TS \pm ST^*) \leq w_{N}(S)\big(N(T) + N(T^*)\big)$.
\end{proof}
%%%%%%%%%%%%%%%%%%%%%%
%%%%%%%%%%%%%%%%%%%%%%
Let us designate the unitary group of all unitary operators in $\mathbb{B}(\mathcal{H})$ by $\mathcal{U}$.
Recall that a norm $N(\cdot)$ on $\mathbb{B}(\mathcal{H})$ is called weakly unitarily invariant
if $N(U^*TU)=N(T)$ for all $T\in\mathbb{B}(\mathcal{H})$ and for all $U\in\mathcal{U}$.
Notice that, if the norm $N(\cdot)$ is weakly unitarily invariant, then so is $w_N(\cdot)$.

Our next result reads as follows.
%%%%%%%%%%%%%%%%%%%%
%%%%%%%%%%%%%%%%%%%%
\begin{theorem}\label{T.1.5.2}
Let $N(\cdot)$ is a weakly unitarily invariant algebra norm on $\mathbb{B}(\mathcal{H})$,
and let $T$ and $S$ be self-adjoint operator in the norm-unit ball of $\mathbb{B}(\mathcal{H})$. Then
\begin{align*}
w_{N}(TS \pm ST) \leq \min\{w_{N}(T), w_{N}(S)\}\sup_{U\in\mathcal{U}}\big\{N(U) + N(U^*)\big\}.
\end{align*}
In particular, if $N(\cdot)$ is a weakly unitarily invariant self-adjoint algebra norm, then
\begin{align*}
w_{N}(TS \pm ST) \leq 2\min\{N(T), N(S)\}\sup_{U\in\mathcal{U}}N(U).
\end{align*}
\end{theorem}
\begin{proof}
Let $U=S+i(I-S^2)^{1/2}$. It follows from the functional calculus for $S$
that $U$ is a unitary operator and $S={\rm Re}\,U$. We have
\begin{align*}
w_{N}(TS \pm ST)&=w_N\left(T\frac{U+U^*}{2}\pm\frac{U+U^*}{2}T\right)
\\&\leq\frac{1}{2}w_N(TU\pm U^*T)+\frac{1}{2}w_N(TU^*\pm UT)
\\&=\frac{1}{2}w_N(TU\pm U^*T)+\frac{1}{2}w_N\big(U^*(UT\pm TU^*)U\big)
\\& \qquad \qquad \qquad \qquad \quad \Big({\rm by~the~weakly~unitary~invariance~of~} w_N(\cdot)\Big)
\\&=w_N(TU\pm U^*T)
\\&\leq w_N(T)\big(N(U) + N(U^*)\big) \qquad \qquad \qquad \qquad \quad({\rm by~Theorem~} \ref{T.1.2})
\\&\leq w_{N}(T)\sup_{U\in\mathcal{U}}\big\{N(U) + N(U^*)\big\}.
\end{align*}
By changing the roles of $T$ and $S$ with each other, we arrive at
\begin{align*}
w_{N}(TS \pm ST)\leq \min\{w_{N}(T), w_{N}(S)\}\sup_{U\in\mathcal{U}}\big\{N(U) + N(U^*)\big\}.
\end{align*}
\end{proof}
%%%%%%%%%%%%%%%%%%%%%%%%%%%
%%%%%%%%%%%%%%%%%%%%%%%%%%%
\begin{corollary}
Let $N(\cdot)$ is a weakly unitarily invariant self-adjoint algebra norm on $\mathbb{B}(\mathcal{H})$
and let $T$ be an operator in the norm-unit ball of $\mathbb{B}(\mathcal{H})$. Then
\begin{align*}
w_{N}(TT^* - T^*T) \leq 4N(T)\sup_{U\in\mathcal{U}}N(U).
\end{align*}
\end{corollary}
\begin{proof}
Let $T =\,{\rm Re}(T) + i\,{\rm Im}(T)$ is the Cartesian decomposition of $T$.
Clearly $TT^*-T^*T=2i\big(\,{\rm Im}(T)\,{\rm Re}(T) - {\rm Re}(T)\,{\rm Im}(T)\big)$.
In addition, ${\rm Re}(T)$ and ${\rm Im}(T)$ are self-adjoint operators in the unit ball of $\mathbb{B}(\mathcal{H})$.
It follows from Theorem \ref{T.1.5.2} that
\begin{align*}
w_{N}(TT^* - T^*T)&=2w_N\big(\,{\rm Im}(T)\,{\rm Re}(T) - {\rm Re}(T)\,{\rm Im}(T)\big)
\\&\leq 4\min\big\{N\big(\,{\rm Im}(T)\big), N\big(\,{\rm Re}(T)\big)\big\}\sup_{U\in\mathcal{U}}N(U)
\\&\leq 4N(T)\sup_{U\in\mathcal{U}}N(U).
\end{align*}
\end{proof}
%%%%%%%%%%%%%%%%%%%%%%%%%%
%%%%%%%%%%%%%%%%%%%%%%%%%%
Our next theorems give some inequalities for the generalized numerical
radius of the product of two Hilbert space operators.
%%%%%%%%%%%%%%%%%%%%%%%%%
%%%%%%%%%%%%%%%%%%%%%%%%%
\begin{theorem}\label{T.3.2}
Let $T, S\in\mathbb{B}(\mathcal{H})$. If $N(\cdot)$ is a self-adjoint algebra norm, then
\begin{align*}
w_{N}(TS) &\leq \min\Big\{N(T)w_{N}(S) + \frac{1}{2} w_{N}(TS \pm ST^*), N(S)w_{N}(T) + \frac{1}{2} w_{N}(TS \pm S^*T)\Big\}
\\& \leq 2\min\Big\{N(T)w_{N}(S), N(S)w_{N}(T)\Big\}
\\& \leq 4w_{N}(T)w_{N}(S).
\end{align*}
\end{theorem}
\begin{proof}
The second inequality follows from Theorem \ref{T.1.2}.
The third inequality follows from (\ref{666}).
It is therefore enough to prove the first inequality.
Let $\theta \in \mathbb{R}$. Since ${\rm Re}\big(e^{i\theta}(TS)\big)$ is self-adjoint, we have
\begin{align*}
N\Big({\rm Re}\big(e^{i\theta}(TS)\big)\Big)&= w_{N}\Big({\rm Re}\big(e^{i\theta}(TS)\big)\Big)
\\& = w_{N}\left(\frac{e^{i\theta}TS + e^{-i\theta}S^*T^*}{2}\right)
\\& = w_{N}\left(T\frac{e^{i\theta}S + e^{-i\theta}S^*}{2} + e^{-i\theta}\frac{S^*T^* - TS^*}{2}\right)
\\& \leq w_{N}\left(T\frac{e^{i\theta}S + e^{-i\theta}S^*}{2}\right) + w_{N}\left(e^{-i\theta}\frac{S^*T^* - TS^*}{2}\right)
\\& \leq N\left(T\frac{e^{i\theta}S + e^{-i\theta}S^*}{2}\right) + w_{N}\left(e^{-i\theta}\frac{S^*T^* - TS^*}{2}\right)
\\& \leq N(T)N\big({\rm Re}(e^{i\theta}S)\big) + \frac{1}{2}w_{N}(S^*T^* - TS^*)
\\& \leq N(T)w_{N}(S) + \frac{1}{2}w_{N}(TS - S^*T)
\end{align*}
Thus
\begin{align}\label{I.1.T.3.2}
w_{N}(TS) = \displaystyle{\sup_{\theta \in \mathbb{R}}}N\Big({\rm Re}\big(e^{i\theta}(TS)\big)\Big)
\leq N(T)w_{N}(S) + \frac{1}{2}w_{N}(TS - S^*T).
\end{align}
Now, by replacing $S$ by $-iS$ in (\ref{I.1.T.3.2}), we obtain
\begin{align}\label{I.2.T.3.2}
w_{N}(TS) &= w_{N}(T(iS))\nonumber\\
&\leq N(T)w_{N}(iS) + \frac{1}{2}w_{N}(T(iS) - (iS)^*T)\nonumber \\
& = N(T)w_{N}(S) + \frac{1}{2}w_{N}(TS + S^*T).
\end{align}
From (\ref{I.1.T.3.2}) and (\ref{I.2.T.3.2}) we conclude that
\begin{align}\label{I.3.T.3.2}
w_{N}(TS) \leq N(T)w_{N}(S) + \frac{1}{2} w_{N}(TS \pm ST^*).
\end{align}
Now, by replacing $T$ by $S^*$ and $S$ by $T^*$ in (\ref{I.3.T.3.2}), we obtain
\begin{align*}
w_{N}(TS) &= w_{N}\big((TS)^*\big) = w_{N}(S^*T^*)
\\& \leq N(S^*)w_{N}(T^*) + \frac{1}{2} w_{N}(S^*T^* \pm T^*S)
\\& = N(S)w_{N}(T) + \frac{1}{2} w_{N}(TS \pm S^*T),
\end{align*}
whence
\begin{align}\label{I.4.T.3.2}
w_{N}(TS) \leq N(S)w_{N}(T) + \frac{1}{2} w_{N}(TS \pm S^*T).
\end{align}
Employing (\ref{I.3.T.3.2}) and (\ref{I.4.T.3.2}) we deduce the desired result.
\end{proof}
%%%%%%%%%%%%%%%%%%%%%%%%%%%%%%%%
%%%%%%%%%%%%%%%%%%%%%%%%%%%%%%%%
We finish this section by the following result.
%%%%%%%%%%%%%%%%%%%%%%%%%%%%%%
%%%%%%%%%%%%%%%%%%%%%%%%%%%%%%
\begin{corollary}\label{C.4.2}
Let $T, S\in\mathbb{B}(\mathcal{H})$ be self-adjoint operators and let $N(\cdot)$ be a self-adjoint algebra norm.
If $TS =\pm ST$, then
\begin{align*}
w_{N}(TS) \leq \min\{N(T)w_{N}(S), N(S)w_{N}(T)\}.
\end{align*}
\end{corollary}
%%%%%%%%%%%%%%%%%%%%%%%%%%
%%%%%%%%%%%%%%%%%%%%%%%%%%
\section{Some applications}
In \cite{D.1}, Dragomir has introduced the following norm on $\mathbb{B}(\mathcal{H})$:
\begin{align*}
\Omega(T) = \sup\Big\{\big\|\zeta T + \eta T^*\big\|: \,\, \zeta, \eta \in \mathbb{C}, \, |\zeta|^2 + |\eta|^2 \leq 1\Big\}
\end{align*}
where $T\in\mathbb{B}(\mathcal{H})$.
It is clear that $\Omega(\cdot)$ is a self-adjoint norm on $\mathbb{B}(\mathcal{H})$ and $\Omega(T) = \sqrt{2}\|T\|$ if $T$ is self-adjoint.
It has been shown in \cite{D.1} that the following equality holds true:
\begin{align}\label{I.1.R.4.20}
\Omega(T) = \displaystyle{\sup_{\|x\|=1, \|y\|=1}} \sqrt{|\langle Ty, x\rangle|^2 + |\langle T^*y, x\rangle|^2}.
\end{align}
As pointed out in \cite{D.1}, $\Omega(\cdot)$ also satisfies the double inequality:
\begin{align}\label{I.1.R.4.2}
\|T\| \leq \Omega(T)\leq \sqrt{2}\|T\|
\end{align}
for each $T\in\mathbb{B}(\mathcal{H})$.
%%%%%%%%%%%%%%%%%%%%%%%%%%
%%%%%%%%%%%%%%%%%%%%%%%%%%
When $N(\cdot)$ is taken to be $\Omega(\cdot)$ on $\mathbb{B}(\mathcal{H})$, the norm $w_{N}(\cdot)$ is denoted by $w_{\Omega}(\cdot)$.
Hence, $w_{\Omega}(T) = \displaystyle{\sup_{\theta\in \mathbb{R}}}\,\Omega\big(\,{\rm Re}(e^{i\theta}T)\big)$
for any $T\in\mathbb{B}(\mathcal{H})$. In the following theorem, we obtain a formula for $w_{\Omega}(T)$ in terms of $w(T)$.
%%%%%%%%%%%%%%%%%%%%%%%%
%%%%%%%%%%%%%%%%%%%%%%%%
\begin{proposition}\label{T.7.2.0}
$w_{\Omega}(T) = \sqrt{2}w(T)$ for all $T\in\mathbb{B}(\mathcal{H})$.
\end{proposition}
\begin{proof}
Let $T\in\mathbb{B}(\mathcal{H})$. From the fact that ${\rm Re}(e^{i\theta}T)$ is self-adjoint for each $\theta\in \mathbb{R}$, we conclude that
\begin{align*}
w_{\Omega}(T) = \displaystyle{\sup_{\theta\in \mathbb{R}}}
\,\Omega\big(\,{\rm Re}(e^{i\theta}T)\big)
= \sqrt{2}\,\displaystyle{\sup_{\theta\in \mathbb{R}}}\,\big\|\,{\rm Re}(e^{i\theta}T)\big\|
= \sqrt{2}w(T).
\end{align*}
\end{proof}
%%%%%%%%%%%%%%%%%%%%%%%%%
%%%%%%%%%%%%%%%%%%%%%%%%%
Because of $\Omega(\cdot)$ is a self-adjoint norm on $\mathbb{B}(\mathcal{H})$,
by (\ref{666}), we observe that
\begin{align*}
\frac{1}{2}\Omega(T)\leq w_{\Omega}(T)\leq \Omega(T).
\end{align*}
In the following theorem, we present an equivalent condition
for $w_{\Omega}(\cdot)=\frac{1}{2}\Omega(\cdot)$.
%%%%%%%%%%%%%%%%%%%%%%%%
%%%%%%%%%%%%%%%%%%%%%%%%
\begin{theorem}\label{T.3.5.2}
Let $T\in\mathbb{B}(\mathcal{H})$. Then the following conditions are equivalent:
\begin{itemize}
\item[(i)] $w_{\Omega}(T) = \frac{1}{2}\Omega(T)$.
\item[(ii)] $\Omega(T) = 2\sqrt{2}\|\,{\rm Re}(e^{i\theta}T)\|$ for all $\theta \in \mathbb{R}$.
\end{itemize}
\end{theorem}
\begin{proof}
Suppose that $w_{\Omega}(T) = \frac{1}{2}\Omega(T)$.
Let $\varphi\in \mathbb{R}$. We have
\begin{align*}
\Omega(T) &= \Omega(e^{i\varphi}T)
= \Omega\big({\rm Re}(e^{i\varphi}T) + i\,{\rm Im}(e^{i\varphi}T)\big)
\\& \leq \Omega\big({\rm Re}(e^{i\varphi}T)\big) + \Omega\big({\rm Im}(e^{i\varphi}T)\big)
\leq w_{\Omega}(T) + w_{\Omega}(T) = \Omega(T),
\end{align*}
which implies
\begin{align}\label{T.I.1.3.5.2}
\Omega\big({\rm Re}(e^{i\varphi}T)\big) + \Omega\big({\rm Im}(e^{i\varphi}T)\big) = \Omega(T).
\end{align}
Furthermore,
\begin{align*}
\frac{1}{2}\Omega(T) &= w_{\Omega}(T)
\\ &\geq \max\Big\{\Omega\big({\rm Re}(e^{i\varphi}T)\big), \Omega\big({\rm Im}(e^{i\varphi}T)\big)\Big\}
\\& = \frac{\Omega\big({\rm Re}(e^{i\varphi}T)\big) + \Omega\big({\rm Im}(e^{i\varphi}T)\big)}{2}
+ \frac{\Big|\Omega\big({\rm Re}(e^{i\varphi}T)\big) - \Omega\big({\rm Im}(e^{i\varphi}T)\big)\Big|}{2}
\\& \geq \frac{\Omega\Big({\rm Re}(e^{i\varphi}T) + i\,{\rm Im}(e^{i\varphi}T)\Big)}{2}
+ \frac{\Big|\Omega\big({\rm Re}(e^{i\varphi}T)\big) - \Omega\big({\rm Im}(e^{i\varphi}T)\big)\Big|}{2}
\\& = \frac{\Omega(e^{i\varphi}T)}{2} + \frac{\Big|\Omega\big({\rm Re}(e^{i\varphi}T)\big) - \Omega\big({\rm Im}(e^{i\varphi}T)\big)\Big|}{2}
\geq \frac{1}{2}\Omega(T),
\end{align*}
which yields
\begin{align}\label{T.I.2.3.5.2}
\Omega\big({\rm Re}(e^{i\varphi}T)\big) = \Omega\big({\rm Im}(e^{i\varphi}T)\big).
\end{align}
Now, by (\ref{T.I.1.3.5.2}) and (\ref{T.I.2.3.5.2}) we conclude that
$\Omega(T) = 2\Omega\big({\rm Re}(e^{i\varphi}T)\big)$. Thus
$\Omega(T) = 2\sqrt{2}\|\,{\rm Re}(e^{i\theta}T)\|$, because ${\rm Re}(e^{i\varphi}T)$ is self-adjoint.

To prove the converse, let $\Omega(T) = 2\sqrt{2}\|\,{\rm Re}(e^{i\theta}T)\|$ for all $\theta \in \mathbb{R}$.
So, by Proposition \ref{T.7.2.0}, we obtain
\begin{align*}
\frac{1}{2}\Omega(T) = \sqrt{2}\,\displaystyle{\sup_{\theta\in \mathbb{R}}}\,\big\|\,{\rm Re}(e^{i\theta}T)\big\|
= \sqrt{2}w(T) = w_{\Omega}(T),
\end{align*}
and hence $w_{\Omega}(T) = \frac{1}{2}\Omega(T)$.
\end{proof}
%%%%%%%%%%%%%%%%%%%%%%%%%
%%%%%%%%%%%%%%%%%%%%%%%%%
In the following theorem, a refinement of the second inequality (\ref{I.1.R.4.2}) is given (see also \cite{D.1}).
%%%%%%%%%%%%%%%%%%%%%%%%
%%%%%%%%%%%%%%%%%%%%%%%%
\begin{theorem}\label{T.5.2.0}
Let $T\in\mathbb{B}(\mathcal{H})$. Then
\begin{align*}
\Omega(T) \leq  \min\left\{\sqrt{\|TT^* + T^*T\|}, \sqrt{\|T\|^2 + w(T^2)}\right\}.
\end{align*}
\end{theorem}
\begin{proof}
We use the following inequality
\begin{align}\label{I.3.T.5.2.001}
|\langle a, c\rangle|^2 + |\langle b, c\rangle|^2\leq \|c\|^2\Big(\max\big\{\|a\|^2, \|b\|^2\big\} + |\langle a, b\rangle|\Big),
\end{align}
for any $a, b, c \in \mathcal{H}$ (see, e.g., \cite{M.P.F}).

Set $a = Ty$, $b = T^*y$, $c = x$, $x, y \in \mathcal{H}$, $\|x\|=\|y\|=1$ in (\ref{I.3.T.5.2.001}) to get
\begin{align*}
|\langle Ty, x\rangle|^2 + |\langle T^*y, x\rangle|^2\leq \|x\|^2\Big(\max\big\{\|Ty\|^2, \|T^*y\|^2\big\} + |\langle Ty, T^*y\rangle|\Big).
\end{align*}
Thus
\begin{align}\label{I.3.T.5.2}
\sqrt{|\langle Ty, x\rangle|^2 + |\langle T^*y, x\rangle|^2}\leq \sqrt{\max\big\{\|Ty\|^2, \|T^*y\|^2\big\} + |\langle T^2y, y\rangle|}.
\end{align}
By taking the supremum over $\|x\|= 1$, $\|y\|= 1$ in (\ref{I.3.T.5.2}), we obtain
\begin{align*}
\displaystyle{\sup_{\|x\|=1, \|y\|=1}} \sqrt{|\langle Ty, x\rangle|^2 + |\langle T^*y, x\rangle|^2}
\leq  \sqrt{\|T\|^2 + w(T^2)}.
\end{align*}
So, by (\ref{I.1.R.4.20}) we arrive at
\begin{align}\label{I.4.T.5.2}
\Omega(T) \leq  \sqrt{\|T\|^2 + w(T^2)}.
\end{align}
Now, let $z \in \mathcal{H}$ with $\|z\|=1$. For every $\zeta, \eta \in \mathbb{C}$ with $|\zeta|^2 + |\eta|^2 \leq 1$,
by the Cauchy--Schwarz inequality we have
\begin{align*}
\big\|\zeta Tz + \eta T^*z\big\|
& \leq \sqrt{|\zeta|^2 + |\eta|^2}\sqrt{\|Tz\|^2 + \|T^*z\|^2}
\\& \leq \sqrt{\big\langle (TT^* + T^*T)z, z\big\rangle} \leq \sqrt{\|TT^* + T^*T\|},
\end{align*}
whence
\begin{align}\label{I.2.T.5.2.07}
\big\|\zeta Tz + \eta T^*z\big\| \leq \sqrt{\|TT^* + T^*T\|}.
\end{align}
Taking the supremum over $\|z\|= 1$ in (\ref{I.2.T.5.2.07}), we obtain
\begin{align*}
\big\|\zeta T + \eta T^*\big\| \leq \sqrt{\|TT^* + T^*T\|},
\end{align*}
and so
\begin{align*}
\sup\Big\{\big\|\zeta T + \eta T^*\big\|: \,\, \zeta, \eta \in \mathbb{C}, \, |\zeta|^2 + |\eta|^2 \leq 1\Big\}
\leq \sqrt{\|TT^* + T^*T\|}.
\end{align*}
Hence
\begin{align}\label{I.2.T.5.2.0}
\Omega(T) \leq \sqrt{\|TT^* + T^*T\|}.
\end{align}
Utilizing (\ref{I.4.T.5.2}) and (\ref{I.2.T.5.2.0}) we deduce the desired result.
\end{proof}
%%%%%%%%%%%%%%%%%%%%%%%%
%%%%%%%%%%%%%%%%%%%%%%%%
The next result refines inequalities (\ref{222}) and (\ref{333}).
%%%%%%%%%%%%%%%%%%%%%%%%%
%%%%%%%%%%%%%%%%%%%%%%%%%
\begin{theorem}\label{T.8.2}
Let $T\in\mathbb{B}(\mathcal{H})$. Then
\begin{align*}
w(T) \leq \frac{\sqrt{2}}{2}\Omega(T) \leq \frac{\sqrt{2}}{2}\min\left\{\sqrt{\|TT^* + T^*T\|}, \sqrt{\|T\|^2 + w(T^2)}\right\}.
\end{align*}
\end{theorem}
\begin{proof}
The second inequality is deduced from Theorem \ref{T.5.2.0}.
It is therefore enough to prove the first inequality.
Since $\Omega(\cdot)$ is a self-adjoint norm on $\mathbb{B}(\mathcal{H})$, it follows from (\ref{666}) that $w_{\Omega}(T) \leq \Omega(T)$.
On the other hand, by Proposition \ref{T.7.2.0}, $w_{\Omega}(T) = \sqrt{2}w(T)$.
Therefore, $\sqrt{2}w(T)\leq \Omega(T)$, or equivalently, $w(T) \leq \frac{\sqrt{2}}{2}\Omega(T)$.
\end{proof}
%%%%%%%%%%%%%%%%%%%%%%%%%%
%%%%%%%%%%%%%%%%%%%%%%%%%%
As an immediate consequence of Theorem \ref{T.8.2}, we have the following result.
%%%%%%%%%%%%%%%%%%%%%%%%%%
%%%%%%%%%%%%%%%%%%%%%%%%%%
\begin{corollary}\label{C.9.2}
If $T\in\mathbb{B}(\mathcal{H})$ is normal, then $\Omega(T) = \sqrt{2}\|T\|$.
\end{corollary}
%%%%%%%%%%%%%%%%%%%%%%
%%%%%%%%%%%%%%%%%%%%%%
\begin{corollary}\label{C.10.2}
Let $T\in\mathbb{B}(\mathcal{H})$. Then the following statements hold.
\begin{itemize}
\item[(i)] If $T$ is normal, then $w_{\Omega}(T) = \Omega(T)$.
\item[(ii)] If $T^2 = 0$, then $w_{\Omega}(T)= \frac{\sqrt{2}}{2}\Omega(T)$.
\end{itemize}
\end{corollary}
\begin{proof}
(i) Let $T$ be normal. Hence $w(T) = \|T\|$.
Furthermore, by Corollary \ref{C.9.2} we have $\Omega(T) = \sqrt{2}\|T\|$.
Therefore, by Proposition \ref{T.7.2.0}, we conclude that
$w_{\Omega}(T) = \sqrt{2}w(T) = \sqrt{2}\|T\| = \Omega(T)$.

(ii) Let $T^2 = 0$. Thus $w(T) = \frac{1}{2}\|T\|$.
Also, from (\ref{I.1.R.4.2}) and (\ref{I.4.T.5.2}) it follows that $\Omega(T) = \|T\|$.
Hence $w_{\Omega}(T) = \sqrt{2}w(T) = \frac{\sqrt{2}}{2}\Omega(T)$.
\end{proof}
%%%%%%%%%%%%%%%%%%%%%%%%%%%%%%%%%%
%%%%%%%%%%%%%%%%%%%%%%%%%%%%%%%%%%
We finish this section by applying our results to obtain some inequalities for Hilbert--Schmidt operators.
Recall that an operator $T\in\mathbb{B}(\mathcal{H})$ is said to belong to the Hilbert--Schmidt class
$\mathcal{C}_2(\mathcal{H})$ if
$\sum_{i = 1}^\infty \|Te_i\|^2 < \infty$, for any orthonormal basis $\{e_i\}^\infty_{i = 1}$ for $\mathcal{H}$.
For $T\in\mathcal{C}_2(\mathcal{H})$, let $\|T\|_2 = \Big(\sum_{i = 1}^\infty \|Te_i\|^2\Big)^{\frac{1}{2}}$
be the Hilbert--Schmidt norm of $T$. It is easy to check that $\|T^*\|_2 = \|T\|_2$
and $\|TS\|_2 \leq \|T\|_2 \|S\|_2$ for all $T, S\in\mathcal{C}_2(\mathcal{H})$,
and so $\|\cdot\|_2$ is a self-adjoint algebra norm on $\mathcal{C}_2(\mathcal{H})$.
When $N(\cdot)$ is the Hilbert--Schmidt norm $\|\cdot\|_2$, the norm $w_N(\cdot)$ is denoted
by $w_2(\cdot)$. For $T\in\mathcal{C}_2(\mathcal{H})$, it was proved in \cite[Theorem 7]{A.K.2} that
\begin{align}\label{I.1.R.7.3}
w_2(T) = \frac{1}{2} \|T\|^2_2 + \frac{1}{2}|{\rm tr}(T^2)|,
\end{align}
where the symbol ${\rm tr}$ denotes the trace functional.

Now, by Theorem \ref{T.5.3} (applied for $N(\cdot) = \|\cdot\|_2$), (\ref{3.00})
and the identity (\ref{I.1.R.7.3}), we achieve our final result.
%%%%%%%%%%%%%%%%%%%%
%%%%%%%%%%%%%%%%%%%%
\begin{corollary}\label{C.4.5.3}
For $T, S\in \mathcal{C}_2(\mathcal{H})$ the following statements hold.
\begin{itemize}
\item[(i)] \begingroup\makeatletter\def\f@size{10}\check@mathfonts
${\big\|TT^* + T^*T\big\|}_2 + \displaystyle{\sup_{\varphi \in \mathbb{R}}}\Big|{\rm tr}\big((e^{i\varphi}T)^2 + (e^{-i\varphi}T^*)^2\big)\Big|
\leq 2\big(\|T\|^2_2 + |{\rm tr}(T^2)|\big).$
\endgroup
\item[(ii)] \begingroup\makeatletter\def\f@size{10}\check@mathfonts
$\|TS\|^2_2 + \big|{\rm tr}(TS)^2\big|
\leq 4\min\Big\{\|T\|^2_2\big(\|S\|^2_2 + |{\rm tr}(S^2)|\big), \|S\|^2_2\big(\|T\|^2_2 + |{\rm tr}(T^2)|\big)\Big\}.$
\endgroup
\end{itemize}
\end{corollary}
%%%%%%%%%%%%%%%%%%%%%%%%%%%%%%%%
%%%%%%%%%%%%%%%%%%%%%%%%%%%%%%%%
\textbf{Acknowledgement.}
Supported by a grant from Shanghai Municipal Science and Technology Commission (18590745200).

%%%%%%%%%%%%%%%%%%%%%%%%%%
%%%%%%%%%%%%%%%%%%%%%%%%%%
\bibliographystyle{amsplain}

\end{document}